\documentclass[envcountsect,envcountsame]{llncs}

\usepackage{amsmath,amssymb,mathtools}
\usepackage{enumitem}
\usepackage{float}
\usepackage[hidelinks]{hyperref}

\newcommand{\Prb}{\mathbb P}
\newcommand{\E}{\mathbb E}
\newcommand{\1}{\mathbf 1}

\title{Multiset Colorings of Random Graphs Across Density Regimes}
\titlerunning{Multiset Colorings of Random Graphs}

\author{
Arash Ahadi\inst{1}\thanks{Corresponding author.}
\and
Sharareh Alipour\inst{1}
}
\authorrunning{A. Ahadi and S. Alipour}
\institute{
Tehran Institute for Advanced Studies (TeIAS),
Khatam University, Iran
\email{aarash.ahadi.academic@gmail.com, sh.alipour@teias.institute}
}

\begin{document}
\pagestyle{empty}
\maketitle
\begin{abstract}
We show that almost every graph admits a partition of its vertex set into three parts such that no two adjacent vertices have the same number of neighbors in each of the three parts. Equivalently, for $G\sim G(n,1/2)$, $\chi_m(G)\le3$ with high probability, improving the previously known bound of five. Here $\chi_m(G)$ denotes the multiset chromatic number of $G$, the minimum number of parts in a vertex partition whose neighbor-count vectors distinguish every pair of adjacent vertices. In fact, the three-part bound holds for every fixed $0.185<p<0.509$. More generally, for every fixed $p\in(0,1)$, $G\sim G(n,p)$ satisfies $\chi_m(G)\le4$ with high probability. These results are obtained by converting the unresolved edges of a carefully chosen initial partition into hyperplanes of a Boolean cube and applying the Linial--Radhakrishnan theory of essential covers.

We also determine how $\chi_m(G)$ grows when the graph is polynomially close to complete. For every fixed $\beta\in(0,1)$ and $G\sim G\!\left(n,1-n^{-(1-\beta)}\right)$, with high probability $\frac{2}{\beta}\le \chi_m(G)\le \left\lfloor\frac{2}{\beta}\right\rfloor+5$. Thus $\chi_m(G)=2/\beta+O(1)$. The lower bound is spectral, while the upper bound follows from multinomial anti-concentration and the Lov'asz Local Lemma.
\end{abstract}

\section{Introduction}
For a graph $G=(V,E)$, let $N_G(v)$ denote the open neighborhood of $v$, and
write $[r]=\{1,\ldots,r\}$ for positive integers $r$. Given a partition
$V=P_1\cup\cdots\cup P_q$, define the
\emph{signature} of a vertex $v$ by
\[
 \sigma(v)=\bigl(|N_G(v)\cap P_i|\bigr)_{i\in[q]}.
\]
The partition is a \emph{multiset coloring} if $\sigma(u)\ne\sigma(v)$ for
every edge $uv\in E(G)$, and the minimum possible number of parts is the
\emph{multiset chromatic number} $\chi_m(G)$
\cite{ChartrandOkamotoSalehiZhang2009}. 

The notion of \emph{multiset coloring} was introduced by Chartrand, Okamoto, Salehi, and Zhang~\cite{ChartrandOkamotoSalehiZhang2009}. Subsequent work treated further graph classes~\cite{OkamotoSalehiZhang2010}, a closely related sum-based variant~\cite{ChartrandOkamotoZhang2010}, powers of cycles~\cite{FengLin2012}, and generalized corona graphs~\cite{FengLin2016}.
Later, Dehghan, Sadeghi, and Ahadi studied the same notion under the name \emph{sigma partitioning}~\cite{DSA18}, obtaining results on its computational complexity and on random graphs. In particular, they proved that for every fixed $p\in(0,1)$, it holds that $\chi_m(G(n,p))\le5$ with high probability.

Although complete graphs show that the vertex version admits no universal
constant bound, random graphs behave strikingly differently. There is a
suggestive parallel with the classical 1--2--3 Conjecture for
edge-weightings. Karo\'nski, {\L}uczak, and Thomason conjectured that weights
$1,2,3$ suffice to distinguish the weighted degrees of adjacent vertices~\cite{KaronskiLuczakThomason2004}. A
landmark intermediate result of Kalkowski, Karo\'nski, and Pfender established
the universal bound five~\cite{KalkowskiKaronskiPfender2010}, and the
conjectured bound three was later proved by Keusch~\cite{Keusch2024}. In our
vertex-partition setting, the previously known bound for $G(n,1/2)$ was
likewise five with high probability~\cite{DSA18}. We first improve the
fixed-density bound from five to four for every fixed $p\in(0,1)$. We then
show that three parts already suffice with high probability for $G(n,1/2)$,
and in fact throughout an explicit interval of fixed edge densities. For every
fixed $p\in(0,1)$ satisfying $ \frac{3\sqrt3\,\pi}{2}\,p(1-p)^2>1$,
we prove
\[
 \Prb \! \bigl(\chi_m(G(n,p))\le3\bigr)\longrightarrow1.
\]
In particular, this condition holds throughout
$0.185<p<0.509$. At $p=1/2$ this improves the previous upper bound
from five to three.

The technical obstacle beyond the previous five-part bound is to eliminate
all residual collisions simultaneously while using only one additional split.
Our proof isolates this obstacle by encoding every unresolved edge as a
hyperplane of a Boolean cube and applying the Linial--Radhakrishnan
essential-cover theorem. The probabilistic part is reduced to showing that
there are only linearly many residual edges, while every corresponding normal
vector has linear support with a strictly larger constant.

Since $\chi_m(K_n)=n$, it is natural to ask how the multiset
chromatic number grows as a random graph approaches completeness.
In the polynomially near-complete regime, we determine its
dependence on the parameter $\beta$ up to an absolute additive
constant, rather than determining its exact value. In particular, for every fixed $\beta\in(0,1)$ and
$p=1-n^{-(1-\beta)}$, with high probability
\[
 \frac2\beta\le\chi_m(G(n,p))
 \le\left\lfloor\frac2\beta\right\rfloor+5.
\]
Thus $\chi_m(G(n,p))=2/\beta+O(1)$. The lower bound is spectral, whereas the
upper bound combines multinomial anti-concentration with the Lov\'asz Local
Lemma.

Section~2 proves the fixed-density bounds and Section~3 treats the near-complete regime.

\section{The Fixed-Density Regime}\label{sec:fixed}

We first prove the three-part bound for an interval of fixed edge densities
and then the four-part bound for every fixed edge density. Throughout this
section $p$ is fixed; constants implicit in $O(\cdot)$ may depend on $p$.

\begin{theorem}\label{thm:half-three}
Fix $p\in(0,1)$ satisfying
$\frac{3\sqrt3\,\pi}{2}p(1-p)^2>1$, and let $G\sim G(n,p)$. Then
\[\Prb\bigl(\chi_m(G)\le3\bigr)\longrightarrow1
 \qquad (n\to\infty).\]
\end{theorem}

The proof has two conceptually separate parts.  First, we give a
deterministic criterion which guarantees that a suitable refinement of a
fixed two-part partition by splitting one of its parts into two produces a multiset coloring. Second, we show
that this criterion holds with high probability in $G(n,p)$.  The main non-elementary structural ingredient is a lemma of Linial and
Radhakrishnan on essential covers of the Boolean cube.

Randomly splitting one part of the initial two-part
partition leaves $\Theta(\sqrt n)$ unresolved edges
in expectation, so a direct first-moment argument
cannot guarantee a successful split.
Our approach instead exploits the structure of the bad
splits: each unresolved edge defines a hyperplane of
the Boolean cube. The Linial--Radhakrishnan support
bound shows that, under our support criterion, these
hyperplanes cannot cover the cube. Thus, even though
a random split need not succeed, at least one split
resolves every edge.

Before exposing the random graph, fix a deterministic set $A\subseteq V$ with $|A|=\lfloor n/4\rfloor$
and put
$$
 B:=V\setminus A,
 \qquad
 |B|=3n/4+O(1).
$$
For $C\in\{A,B\}$ and $v\in V$, write
$$
 d_C(v):=|N_G(v)\cap C|.
$$
Define an auxiliary graph $H$ on the same vertex set by
\begin{equation}\label{eq:half-H-def}
uv\in E(H)
\quad\Longleftrightarrow\quad
\Bigl(uv\in E(G),\quad d_A(u)=d_A(v),\quad d_B(u)=d_B(v)\Bigr).
\end{equation}
Thus, if $uv\in E(G)\setminus E(H)$, the initial two-part partition
$(A,B)$ already distinguishes the endpoints in at least one of the two
coordinates $d_A,d_B$. Consequently, after $B$ is refined, only the
edges of $H$ will need to be repaired.

For vertices $x,y\in V$, let
\[
 A_{xy}:=\mathbf1_{\{xy\in E(G)\}},
\]
where $A_{xx}:=0$.
$$
 a^{uv}_w:=A_{uw}-A_{vw}\qquad(w\in B),
$$
$$
 S_{uv}:=|\operatorname{supp}(a^{uv})|
 =|\{w\in B:A_{uw}\ne A_{vw}\}|.
$$
We shall use the following lemma of Linial and
Radhakrishnan~\cite{LinialRadhakrishnan2005}; the same statement is
recorded explicitly as Lemma~13 in
Araujo--Balogh--Mattos~\cite{AraujoBaloghMattos2025}.

\begin{lemma}[Linial--Radhakrishnan]\label{lem:half-LR}
Suppose that $k$ affine hyperplanes
$$
 \langle v_i,x\rangle=\mu_i,
 \qquad i\in[k],
$$
form an \emph{essential cover} of the Boolean cube $\{0,1\}^m$, meaning
that
\begin{enumerate}[label=(\roman*)]
\item every vertex of the cube lies on at least one of the hyperplanes;
\item no hyperplane is redundant, that is, deleting any one of them
leaves some vertex of the cube uncovered;
\item every coordinate $x_j$ occurs with a nonzero coefficient in the
equation of at least one hyperplane.
\end{enumerate}
Then, for every $i$,
\begin{equation}\label{eq:half-LR-support}
|\operatorname{supp}(v_i)|\le2k.
\end{equation}
\end{lemma}

The lemma applies to arbitrary real affine hyperplanes, and hence in
particular to the homogeneous hyperplanes used below. Indeed, condition (iii) and the support bound imply
\[
m\le \sum_{i=1}^k |\operatorname{supp}(v_i)|\le 2k^2;
\]
so, the existence of an essential cover by $k$ hyperplanes
already requires $m\le 2k^2$.

We can now isolate the deterministic core of the proof.

\begin{lemma}[Hyperplane criterion]
\label{lem:half-hyperplane-criterion}
Suppose that
\begin{equation}\label{eq:half-support-vs-M}
S_{uv}>2|E(H)|
\qquad\text{for every }uv\in E(H).
\end{equation}
Then $B$ can be partitioned into two sets $B_1,B_2$ such that
$(A,B_1,B_2)$ is a multiset coloring of $G$.
\end{lemma}

\begin{proof}
For $x\in\{0,1\}^{B}$ put
$$
 B^x_1=\{w\in B:x_w=1\},
 \qquad
 B^x_2=B\setminus B_1^x.
$$
For $uv\in E(H)$,
\begin{equation}\label{eq:half-dot-product}
\langle a^{uv},x\rangle
=d_{B_1^x}(u)-d_{B_1^x}(v),
\end{equation}
where $\langle ., . \rangle$ is the standard inner product of two vectors. Since $d_B(u)=d_B(v)$ for every $uv\in E(H)$, the edge $uv$ remains
unresolved after the refinement if and only if
\begin{equation}\label{eq:half-bad-iff-hyperplane}
\langle a^{uv},x\rangle=0.
\end{equation}
Thus it suffices to find a point of $\{0,1\}^{B}$ outside all the
hyperplanes
$$
 \mathcal H_{uv}
 =
 \{x\in\mathbb R^{B}:\langle a^{uv},x\rangle=0\},
 \qquad uv\in E(H).
$$
Suppose, to the contrary, that these hyperplanes cover
$\{0,1\}^{B}$. Choose an inclusion-minimal subcover with $k$ members.
Remove every coordinate whose coefficient is zero in every equation
defining a hyperplane of this subcover. After this projection the subcover
remains a cover, giving condition (i) of Lemma~\ref{lem:half-LR}; minimality
gives condition (ii), and deletion of unused coordinates gives condition (iii). Hence, since
$k\le |E(H)|$, Lemma~\ref{lem:half-LR} gives
$$
 |\operatorname{supp}(a^{uv})|
 \le2k\le2|E(H)|
$$
for every hyperplane in the subcover. This contradicts
\eqref{eq:half-support-vs-M}. Hence there exists
\begin{equation}\label{eq:half-good-x}
x\in\{0,1\}^{B}
\qquad\text{such that}\qquad
\langle a^{uv},x\rangle\ne0
\quad\text{for every }uv\in E(H).
\end{equation}

Choose such an $x$, and define $B_1,B_2$ accordingly. Let
$uv\in E(G)$. If $d_A(u)\ne d_A(v)$, then $A$ distinguishes the
endpoints. If $d_A(u)=d_A(v)$ but $d_B(u)\ne d_B(v)$, then at least one
of $B_1,B_2$ distinguishes them, since
$$
 d_B=d_{B_1}+d_{B_2}.
$$
Finally, if
$$
 d_A(u)=d_A(v)
 \qquad\text{and}\qquad
 d_B(u)=d_B(v),
$$
then $uv\in E(H)$ and
$$
 d_{B_1}(u)-d_{B_1}(v)
 =\langle a^{uv},x\rangle\ne0
$$
by \eqref{eq:half-dot-product} and \eqref{eq:half-good-x}. Thus
$(A,B_1,B_2)$ is a multiset coloring.

If nonempty parts are required, note that when $H\ne\varnothing$ both
$0$ and $\mathbf1$ lie on every $\mathcal H_{uv}$, so the chosen point
outside the union is nonconstant and both $B_1$ and $B_2$ are nonempty.
If $H=\varnothing$, the initial partition $(A,B)$ already works and
$B$ may be refined into two nonempty parts. \qed
\end{proof}

\paragraph{\textbf{Choice of the partition sizes.}}
The ratio $|A|:|B|=1:3$ reflects a trade-off:
a larger $B$ gives the normal vectors larger supports,
whereas a more balanced partition reduces the number
of unresolved edges. To quantify this trade-off,
consider $|A|=\alpha n+O(1)$ and
$|B|=(1-\alpha)n+O(1)$ for fixed $\alpha\in(0,1)$.

By the local collision estimate used below, the
probabilities that two vertices have equal degrees
into $A$ and into $B$ are asymptotically
\[
\frac{1}{\sqrt{4\pi\alpha n p(1-p)}}
\quad\text{and}\quad
\frac{1}{\sqrt{4\pi(1-\alpha)n p(1-p)}},
\]
respectively. Accounting for the probability $p$
that the pair is an edge, the same second-moment
argument as below gives, with high probability,
\[
|E(H)|=
\left(
\frac{1}{8\pi(1-p)\sqrt{\alpha(1-\alpha)}}+o(1)
\right)n.
\]
On the other hand, each coordinate of a normal
vector on $B$ is nonzero with probability
$2p(1-p)$. Thus, by concentration and a union bound,
uniformly over distinct $u,v$,
\[
S_{uv}=
\left(2p(1-p)(1-\alpha)+o(1)\right)n
\]
with high probability.

Consequently, the criterion $S_{uv}>2|E(H)|$
is asymptotically satisfied whenever
\[
8\pi p(1-p)^2(1-\alpha)
\sqrt{\alpha(1-\alpha)}>1.
\]
For fixed $p$, maximizing the left-hand side
amounts to maximizing $\alpha(1-\alpha)^3$,
whose derivative is $(1-\alpha)^2(1-4\alpha)$.
The unique maximum occurs at $\alpha=1/4$,
so the ratio $1:3$ is optimal for this strategy.
Substitution yields
\[
\frac{3\sqrt3\,\pi}{2}\,p(1-p)^2>1,
\]
precisely the hypothesis of Theorem~\ref{thm:half-three},
which holds for $0.185<p<0.509$.

The probabilistic part of the argument is summarized in the following
lemma.

\begin{lemma}[Random estimates]\label{lem:half-random-estimates}
With high probability,
\begin{equation}\label{eq:half-M-concentration}
|E(H)|=
\left(
\frac{1}{2\pi\sqrt3\,(1-p)}+o(1)
\right)n,
\end{equation}
and simultaneously
\begin{equation}\label{eq:half-support-bound}
\min_{u\ne v}S_{uv}
\ge
\left(
\frac32p(1-p)-o(1)
\right)n.
\end{equation}
\end{lemma}

We first finish the proof of the theorem from these two estimates.

{
\renewcommand{\proofname}{Proof of Theorem~\ref{thm:half-three}}
\begin{proof}
Put
$$
 c_p:=\frac{1}{2\pi\sqrt3\,(1-p)}.
$$
By Lemma~\ref{lem:half-random-estimates},
$$
 |E(H)|=(c_p+o(1))n,
$$
while, simultaneously for every $uv\in E(H)$,
$$
 S_{uv}
 \ge
 \left(
 \frac32p(1-p)-o(1)
 \right)n.
$$
The decisive strict inequality is
\begin{equation}\label{eq:half-numerical-gap}
 \frac{3}{2}p(1-p)-2c_p
 =
 \frac{3}{2}p(1-p)
 -
 \frac{1}{\pi\sqrt{3}\,(1-p)}
 >0,
\end{equation}
where the last inequality is precisely the hypothesis of the theorem.
Since the gap in \eqref{eq:half-numerical-gap} is a fixed positive
constant, the $o(1)$ terms above are negligible for all sufficiently
large $n$.
Therefore by Lemma \ref{lem:half-random-estimates}, with high probability,
\[
 S_{uv}>2|E(H)|
 \qquad\text{for every }uv\in E(H).
\]
Lemma~\ref{lem:half-hyperplane-criterion} therefore gives a partition
$(A,B_1,B_2)$ which is a multiset coloring of $G$. Hence
\[
 \Prb\bigl(\chi_m(G)\le3\bigr)\longrightarrow1.
\] \qed

\end{proof}
}

\subsection*{Proof of the random estimates}

We now prove Lemma~\ref{lem:half-random-estimates}. The calculation of
the number of unresolved edges uses the following local collision
estimate.

\begin{lemma}\label{lem:half-qm}
Let $X_1,\ldots,X_m,Y_1,\ldots,Y_m$ be independent
Bernoulli random variables with mean $p$ and put
$$
 S_m:=\sum_{i=1}^m(X_i-Y_i),
 \qquad
 q_m(t):=\Prb(S_m=t).
$$
Uniformly for $t$ in any fixed finite subset of $\mathbb Z$,
\begin{equation}\label{eq:half-qm-asymp}
q_m(t)
=
\frac{1+o(1)}{\sqrt{4\pi m p(1-p)}}.
\end{equation}
\end{lemma}

\begin{proof}
Put $Z_i=X_i-Y_i$. Then the $Z_i$ are i.i.d.\ lattice random variables
with mean $0$, variance $2p(1-p)$, and maximal span $1$, since their
support is $\{-1,0,1\}$. 
The lattice local central limit theorem (see, e.g.,~\cite{Petrov1975})
states that if $Z_1,Z_2,\ldots$ are i.i.d.\ integer-valued random
variables with mean $0$, variance $\sigma^2>0$, and maximal span $1$,
then, for every fixed $t\in\mathbb Z$,
\[
 \sqrt m\,\Prb\!\left(\sum_{i=1}^m Z_i=t\right)
 \longrightarrow \frac{1}{\sqrt{2\pi\sigma^2}}.
\]
Here $\sigma^2=2p(1-p)$, and hence
\[
 q_m(t)
 =
 \frac{1+o(1)}{\sqrt{4\pi m p(1-p)}}.
\]
This proves \eqref{eq:half-qm-asymp}. Uniformity over any
fixed finite set follows from pointwise convergence. \qed
\end{proof}

{
\renewcommand{\proofname}{Proof of Lemma~\ref{lem:half-random-estimates}}
\begin{proof} We first establish the estimate for $|E(H)|$. For every unordered pair
$\{u,v\}$, define
$$
 I_{uv}=
\begin{cases}
1, & uv\in E(H),\\
0, & uv\notin E(H),
\end{cases},
 \qquad\qquad
 |E(H)|=\sum_{u<v}I_{uv}
$$
Uniformly over distinct $u,v$,
\begin{equation}\label{eq:half-one-edge}
\Prb(I_{uv}=1)
=
\frac{1+o(1)}{\pi\sqrt{3}\,(1-p)n}.
\end{equation}
Indeed, condition on $uv\in E(G)$, which has probability $p$.
For $C\in\{A,B\}$,
$$
 d_C(u)-d_C(v) =
 \sum_{w\in C\setminus\{u,v\}} (A_{uw}-A_{vw}) + A_{uv}
 \bigl(
 \mathbf1_{\{v\in C\}} - \mathbf1_{\{u\in C\}}\bigr).
$$
After conditioning on $A_{uv}=1$, the last term is a fixed element of
$\{-1,0,1\}$. Thus $d_C(u)=d_C(v)$ asks a sum from
Lemma~\ref{lem:half-qm} to hit a bounded target. The numbers of summands
are
$$
 m_A=n/4+O(1),
 \qquad
 m_B=3n/4+O(1),
$$
and the edge variables used in the two parts are disjoint. Hence
\[
\begin{aligned}
 \Prb(I_{uv}=1)
 &=
 p\,
 q_{m_A}\!\left(
 \mathbf1_{\{u\in A\}}-\mathbf1_{\{v\in A\}}
 \right)
 q_{m_B}\!\left(
 \mathbf1_{\{u\in B\}}-\mathbf1_{\{v\in B\}}
 \right)
 \\
 &=
 \frac{1}{4\pi(1-p)\sqrt{m_A m_B}}
 \bigl(1+o(1)\bigr)
 \\
 &=
 \frac{1+o(1)}{\pi\sqrt3\,(1-p)n}.
\end{aligned}
\]
which proves \eqref{eq:half-one-edge}. Summing over the
$\binom n2$ unordered pairs gives
$$
 \E|E(H)|
 =
 \left(
 \frac{1}{2\pi\sqrt3\,(1-p)}
 +o(1)
 \right)n.
$$
For the variance, it is enough to prove
$\operatorname{Var}(|E(H)|)=o(n^2)$. The diagonal contribution is $O(n)$.
Pairs sharing one endpoint have covariance $O(n^{-2})$ by conditioning
on their three endpoints and using the $O(n^{-1/2})$ maximum binomial
point mass, so their total contribution is $O(n)$. 
Indeed, after the three endpoints are exposed, in each of $A$ and $B$
the two required degree equalities involve three independent binomial
variables, and their joint probability is $O(n^{-1})$.
For four distinct
endpoints, expose the four cross edges. Conditional on them, the two
indicators are independent, while Lemma~\ref{lem:half-qm} gives each
conditional probability as
\[
 \frac{1+o(1)}{\pi\sqrt3\,(1-p)n}
\]
uniformly over the finitely many exposed configurations. Hence each such
covariance is $o(n^{-2})$, and their total contribution is $o(n^2)$.
Thus $\operatorname{Var}(|E(H)|)=o(n^2)$. Chebyshev's inequality now gives
$$
 |E(H)|=
 \left(
 \frac{1}{2\pi\sqrt3\,(1-p)}
 +o(1)
 \right)n
$$
with high probability.

It remains to prove the support estimate. The important point is that
we estimate $S_{uv}$ simultaneously for all pairs $u,v$, without
conditioning on $uv\in E(H)$. Thus no change of distribution caused by
the condition $d_B(u)=d_B(v)$ enters the argument.

For fixed $u\ne v$ and every $w\in B\setminus\{u,v\}$,
$$
 \Prb(A_{uw}\ne A_{vw})
 =
 2p(1-p),
$$
independently over $w$. Hence
\[
 S_{uv}
 \sim
 \operatorname{Binomial}
 \bigl(
 |B\setminus\{u,v\}|,
 2p(1-p)
 \bigr)
 +O(1).
\]
Since $|B|=3n/4+O(1)$, the binomial term has mean
$$
 \left(
 \frac32p(1-p)+o(1)
 \right)n.
$$
Consequently, taking $\varepsilon=n^{-1/4}$, Chernoff's inequality gives
$$
 \Prb\!\left(
 S_{uv}
 <
 \left(
 \frac32p(1-p)-n^{-1/4}
 \right)n
 \right)
 \le
 e^{-\Omega_p(n^{1/2})}.
$$
A union bound over all fewer than $n^2/2$ pairs gives, with high
probability,
$$
 S_{uv}
 \ge
 \left(
 \frac32p(1-p)-n^{-1/4}
 \right)n
$$
simultaneously for all distinct $u,v$. Therefore
$$
 \min_{u\ne v}S_{uv}
 \ge
 \left(
 \frac32p(1-p)-o(1)
 \right)n
$$
with high probability. This proves Lemma~\ref{lem:half-random-estimates}. \qed \end{proof} }

By starting from the partition $(A_1,A_2,B)$ instead of $(A,B)$, we obtain the upper bound $4$ for every $p\in(0,1)$.

\begin{theorem}\label{thm:fixed-four}
For every fixed $p\in(0,1)$ and $G\sim G(n,p)$,
$\Prb(\chi_m(G)\le4)\to1$.
\end{theorem}

\begin{proof}
Partition $V=A_1\cup A_2\cup B$ with $|A_1|=|A_2|=n/4+O(1)$ and
$|B|=n/2+O(1)$, and let $H$ consist of the edges unresolved by
$(A_1,A_2,B)$. Conditioning on $uv\in E(G)$, the three degree collisions
have probability $O(n^{-1/2})$ each and use disjoint remaining edge variables,
so $\Prb(uv\in E(H))=O(n^{-3/2})$. Hence $\E|E(H)|=O(n^{1/2})$ and,
by Markov's inequality, $|E(H)|=o(n)$ with high probability. Meanwhile, Chernoff's inequality and a
union bound give, simultaneously for all $u\ne v$,
$|\{w\in B:A_{uw}\ne A_{vw}\}|\ge(p(1-p)-o(1))n$.
Thus the support exceeds $2|E(H)|$ on every edge of $H$, and the proof of
Lemma~\ref{lem:half-hyperplane-criterion} applies verbatim, with $A_1,A_2$
as the fixed parts, to split $B$ into $B_1,B_2$ and obtain the multiset
coloring $(A_1,A_2,B_1,B_2)$. \qed
\end{proof}

\section{The near-complete regime}\label{sec:near}

We now turn to the near-complete regime and determine how the multiset chromatic number grows as the graph approaches completeness.

\begin{theorem}\label{thm:near}
Fix $\beta\in(0,1)$ and let $G\sim G\!\left(n,1-n^{-(1-\beta)}\right)$.

Then, with high probability,
\[
 \frac{2}{\beta}\le \chi_m(G)
 \le \left\lfloor\frac{2}{\beta}\right\rfloor+5.
\]
Consequently, $ \chi_m(G)=\frac{2}{\beta}+O(1)$.
\end{theorem}

For the proof, put
\[
 H=\overline G\sim G(n,r),\qquad
 r=n^{-(1-\beta)}.
\]
For any partition $V=P_1\cup\cdots\cup P_q$ and every $v\in V$, writing
$N_H[v]:=N_H(v)\cup\{v\}$ for the closed neighborhood in $H$,
\begin{equation}\label{eq:complement-signature}
 |N_G(v)\cap P_i|
 =|P_i|-|N_H[v]\cap P_i|.
\end{equation}
Thus two adjacent vertices of $G$ have the same signature exactly when their closed-neighborhood color-count vectors in $H$ are equal. In particular, vertices of unequal $H$-degree are automatically distinguished.

We use the following lemmas in the proof of Theorem~\ref{thm:near}.

Let $A_H$ be the adjacency matrix of $H$, and let $J$ and $I$ denote
the all-ones and identity matrices. Then $\mathbb E[A_H]=r(J-I)$.
Write $\|M\|=\|M\|_{2\to2}$ for the spectral norm of a real symmetric
matrix $M$.

\begin{lemma}[Lu--Peng spectral estimate. \cite{LuPeng2013}]\label{lem:lu-peng}
Let $H\sim G(n,r)$ and suppose that \((n-1)r\gg(\log n)^4\).
Then, with high probability,
$$
\|A_H-r(J-I)\|=O\!\left(\sqrt{(n-1)r}\right).
$$
\end{lemma}

We use the symmetric Lov\'asz Local Lemma in the standard form
$ep(D+1)\le1$; see Appendix~\ref{app:lll}.

We shall also use the standard local limit theorem for the multinomial
distribution; see, e.g., Ouimet~\cite{Ouimet2021}. Let
$X=(X_1,\ldots,X_q)\sim\operatorname{Mult}(s;1/q,\ldots,1/q)$,
where this denotes the color-count vector obtained from $s$ independent
uniform choices from $[q]$. The theorem implies that, for fixed $q\ge2$,
\[
 \sup_{\substack{(a_1, \ldots, a_q)\in\mathbb Z_{\ge0}^q\\[1pt]\sum_{i=1}^q a_i=s}}
 \Prb(X=(a_1, \ldots, a_q))
 =
 O_q\!\left(s^{-(q-1)/2}\right).
\]
\begin{lemma}[Typical complement properties]\label{lem:typical-complement}
With high probability, $H$ satisfies simultaneously (with $\Delta$ and $\omega$
denoting maximum degree and clique number):
\begin{enumerate}[label=(\roman*)]
\item $\Delta(H)=O(n^\beta)$;
\item uniformly for distinct $u,v$,
$$
 |N_H[u]\triangle N_H[v]|=(2+o(1))n^\beta;
$$
in particular, if $d_H(u)=d_H(v)$, then
$$|N_H[u]\setminus N_H[v]|
 =|N_H[v]\setminus N_H[u]|
 =(1+o(1))n^\beta;$$
\item every degree class of $H$ has size $O(n^{1-\beta/2})$;
\item $\omega(H)=O_\beta(1)$.
\end{enumerate}
\end{lemma}

\begin{proof} All parts are proved using only elementary probabilistic tools.
Parts (i)--(ii) follow from Chernoff bounds. For (iii), split the
vertex set into two halves; conditional on the graph induced by one
half, the contributions from the other half are independent binomial
variables with maximum point mass $O(n^{-\beta/2})$, and a Chernoff
bound gives degree classes of size $O(n^{1-\beta/2})$. Part (iv)
follows from a first-moment bound for $K_t$. Complete details are given
in Appendix~\ref{app:typical-complement}. \qed
\end{proof}

We are now ready to prove Theorem \ref{thm:near}.

{
\renewcommand{\proofname}{Proof of Theorem~\ref{thm:near}}
\begin{proof}
\medskip
\noindent

\underline{\emph{Lower bound.}}
By Lemma~\ref{lem:lu-peng}, since
$(n-1)r=(1+o(1))n^\beta\gg(\log n)^4$,
\[
 \|A_H-r(J-I)\|=O(n^{\beta/2})
\]
with high probability. Put
\[
 C:=A_H+I-rJ.
\]
Then
\begin{equation}\label{eq:gamma-bound}
\begin{aligned}
 \gamma:=\|C\|
 &\le \|A_H-r(J-I)\|+(1-r)\\
 &=O(n^{\beta/2})+\bigl(1-n^{-(1-\beta)}\bigr)
 =O(n^{\beta/2}).
\end{aligned}
\end{equation}
Condition on this event and on the event in
Lemma~\ref{lem:typical-complement}(iv). Suppose that
$V=P_1\cup\cdots\cup P_q$ is a multiset coloring of $G$, and let
$x_i\in\{0,1\}^n$ be the indicator vector of $P_i$.
We first show that, on at least half of the vertices, all the coordinates
$(Cx_i)_v$ are small. Since $\gamma=\|C\|_{2\to2}$, for every $i$,
$\|Cx_i\|_2\le\gamma\|x_i\|_2$. Hence
\begin{equation}\label{eq:spectral-energy}
 \sum_{i=1}^q\|Cx_i\|_2^2
 \le \gamma^2\sum_{i=1}^q\|x_i\|_2^2
 =\gamma^2n.
\end{equation}
Here $\|x_i\|_2^2=|P_i|$ and $\sum_i|P_i|=n$.

For each vertex $v$, let
$y_v=\sum_{i=1}^q(Cx_i)_v^2$.
By \eqref{eq:spectral-energy},
\[
\frac1n\sum_{v\in V}y_v
=\frac1n\sum_{i=1}^q\|Cx_i\|_2^2
\le\gamma^2.
\]
Since $y_v\ge0$, at least half of the vertices
satisfy $y_v\le2\gamma^2$. Thus, the set
\[
W=\{v\in V:y_v\le2\gamma^2\}
\]
has size at least $n/2$. Moreover, for every $v\in W$ and every $i$,
\[
 |(Cx_i)_v|^2
 \le \sum_{j=1}^q(Cx_j)_v^2
 =y_v
 \le2\gamma^2,
\]
and hence $|(Cx_i)_v|\le\sqrt2\,\gamma$.

Now define
$h_i(v)=|N_H[v]\cap P_i|$. Since $A_H+I=rJ+C$, we have
\[
 h_i(v)=((A_H+I)x_i)_v=r|P_i|+(Cx_i)_v.
\]
Thus all possible integer values of $h_i(v)$ lie in an interval of
length at most $2\sqrt2\,\gamma$. Hence $h_i(v)$ takes at most
$2\sqrt2\,\gamma+1$ distinct values on $W$, and the full vector
\[
 h(v)=(h_1(v),\ldots,h_q(v))
\]
takes at most $(2\sqrt2\,\gamma+1)^q$
distinct values on $W$.

By \eqref{eq:complement-signature}, two vertices with the same $h$-vector
have the same multiset signature in $G$. Since
$P_1,\ldots,P_q$ is a multiset coloring, two adjacent vertices of $G$
cannot have the same signature. Consequently, every fiber
$\{v\in W:h(v)=z\}$ is an independent set in $G$, and hence a clique in
$H=\overline G$. By Lemma~\ref{lem:typical-complement}(iv), each such
fiber has size at most $\omega(H)=O_\beta(1)$.

Since $W$ is the union of at most $(2\sqrt2\,\gamma+1)^q$ such fibers,
\[
 \frac n2
 \le |W|
 \le \omega(H)(2\sqrt2\,\gamma+1)^q
 =O_\beta(1)O(n^{\beta q/2})
 =O_{\beta,q}\!\left(n^{\beta q/2}\right),
\]
where we used $\gamma=O(n^{\beta/2})$; see (\ref{eq:gamma-bound}). Here $q$ is a fixed integer;
since $\beta$ is fixed, there are only finitely many integers
$q<2/\beta$, so $q$ may be treated as a constant.

If $q<2/\beta$, then $\beta q/2<1$, and therefore the right-hand side is
$o(n)$, contradicting the left-hand side $n/2$. Thus no multiset
coloring with an integer number $q<2/\beta$ of parts exists, and
consequently
\[
 \chi_m(G)\ge\left\lceil\frac{2}{\beta}\right\rceil
 \ge\frac{2}{\beta}.
\]

\medskip
\noindent
\underline{\emph{Upper bound.}}
Set $q=\left\lfloor\frac2\beta\right\rfloor+5$.
By Lemma~\ref{lem:typical-complement}, with high probability $H$
satisfies properties (i)--(iv). Fix any realization of $H$ satisfying
these properties, and color its vertices independently and uniformly
from $[q]$.

Only pairs $uv\notin E(H)$ can be bad, since these are precisely the
edges of $G$. Moreover, if $u$ and $v$ have equal closed-neighborhood
color-count vectors, then
\[
 d_H(u)+1=\sum_i h_i(u)=\sum_i h_i(v)=d_H(v)+1,
\]
so necessarily $d_H(u)=d_H(v)$. Hence only pairs
$uv\notin E(H)$ with $d_H(u)=d_H(v)$ need to be considered. For every such pair, put
\[
 U_{uv}=N_H[u]\setminus N_H[v],
 \qquad
 V_{uv}=N_H[v]\setminus N_H[u].
\]
By Lemma~\ref{lem:typical-complement}(ii),
\[
 |U_{uv}|=|V_{uv}|=(1+o(1))n^\beta
\]
uniformly over all relevant pairs.

Let $\mathcal B_{uv}$ be the event that $u$ and $v$ have equal
closed-neighborhood color-count vectors. The common part
$N_H[u]\cap N_H[v]$ contributes equally to the two vectors, so $\mathcal B_{uv}$
occurs exactly when the color-count vectors on $U_{uv}$ and $V_{uv}$
are equal. Since these sets are disjoint, the two vectors are
independent multinomial random vectors with
$s=(1+o(1))n^\beta$ trials and $q$ equally likely colors. Let $X$ and $Y$
denote these two vectors. Since $X$ and $Y$ are independent and identically distributed, by the multinomial local limit theorem quoted
above,
\[
 \sup_a\Prb(X=a)=O_q\!\left(s^{-(q-1)/2}\right).
\]
Hence, uniformly over all relevant pairs,
\[
\begin{aligned}
 \Prb(\mathcal B_{uv})
 &=\Prb(X=Y)
 =\sum_a\Prb(X=a)^2\\
 &\le
 \left(\sup_a\Prb(X=a)\right)\sum_a\Prb(X=a)
 =O_q\!\left(s^{-(q-1)/2}\right)
 =O_q\!\left(n^{-\beta(q-1)/2}\right).
\end{aligned}
\]
It remains to control dependencies. The event $\mathcal B_{uv}$ depends only on
the colors in
\[
 W_{uv}=N_H[u]\triangle N_H[v].
\]
Define a graph on the bad events by joining $\mathcal B_{uv}$ and
$\mathcal B_{xy}$ whenever $W_{uv}\cap W_{xy}\ne\varnothing$. Since the
vertex colors are independent, this is a dependency graph for the family of
bad events. By Lemma~\ref{lem:typical-complement}(i),(iii), a fixed vertex
$w$ belongs to $W_{xy}$ for at most
\[
 |N_H[w]|\cdot\max_t|\{v:d_H(v)=t\}|=O(n^{1+\beta/2})
\]
relevant pairs $xy$: if $w\in N_H[x]\triangle N_H[y]$, exactly one endpoint lies
in $N_H[w]$, and the other has the same $H$-degree. Since
$|W_{uv}|=O(n^\beta)$ by Lemma~\ref{lem:typical-complement}(ii), every
$\mathcal B_{uv}$ is therefore adjacent in the dependency graph to at most
\[
 D=O(n^\beta)\,O(n^{1+\beta/2})=O(n^{1+3\beta/2})
\]
other bad events.
Consequently,
\[
 \Prb(\mathcal B_{uv})(D+1)
 =
 O_q\!\left(n^{\,1+\beta(4-q)/2}\right).
\]
Our choice
$q=\lfloor2/\beta\rfloor+5$ satisfies
$q>4+2/\beta$, and hence the exponent is negative, that is
\[
 1+\frac{\beta(4-q)}2<0.
\]
Therefore
\[
 e\,\Prb(\mathcal B_{uv})(D+1)=o(1),
\]
so the hypothesis of the symmetric Lov\'asz Local Lemma holds for all
sufficiently large $n$. Hence there exists a coloring for which no bad
event occurs.

By \eqref{eq:complement-signature}, this coloring is a multiset coloring
of $G$, proving $\chi_m(G)\le\lfloor2/\beta\rfloor+5$.
Moreover, the Moser--Tardos resampling algorithm~\cite{MoserTardos2010}
finds such a coloring in expected polynomial time. \qed
\end{proof}
}

\section{Concluding Remarks and Open Problems}
Our results leave open the exact multiset chromatic number
of random graphs with fixed edge probability. All the
following questions concern properties holding with high
probability as $n\to\infty$.

In particular, is
$\chi_m(G(n,1/2))=2$?
More generally, is
$\chi_m(G(n,p))\le3$
for every fixed $p\in(0,1)$?
Can the stronger equality
$\chi_m(G(n,p))=2$
be established for every such $p$?

\subsubsection*{Declaration of AI Use}
Claude and, in particular, ChatGPT were used for idea
generation, developing and checking proofs, and writing.
All AI-assisted contributions were reviewed and verified
by the human authors, who take full responsibility
for the content of this paper.

\appendix
\makeatletter
\renewcommand*{\theHsection}{app.\Alph{section}}
\renewcommand*{\theHsubsection}{\theHsection.\arabic{subsection}}
\makeatother

\section{Technical details for the fixed-density regime}
\label{app:half-three-supplement}

\subsection{Covariance details for the unresolved-edge count}
\label{app:half-variance}

We give here the covariance calculation omitted from the proof of
Lemma~\ref{lem:half-random-estimates}.

It remains to control the variance. Write
\[
 \operatorname{Var}(|E(H)|)
 =
 \sum_e\operatorname{Var}(I_e)
 +2\sum_{e<f}\operatorname{Cov}(I_e,I_f).
\]
The diagonal contribution is immediate:
\[
 \sum_e\operatorname{Var}(I_e)
 \le \sum_e\E I_e
 =O(n).
\]
Suppose next that two pairs share one endpoint, say $e=uv$ and $f=uw$.
Condition on the three edge indicators inside $\{u,v,w\}$. For a fixed
part $C\in\{A,B\}$, after deleting $u,v,w$, the numbers of neighbors of
these three vertices in $C\setminus\{u,v,w\}$ are independent
$\operatorname{Bin}(m_C,p)$ variables $X,Y,Z$, with $m_C=\Theta(n)$.
The two required degree equalities become
\[
 X-Y=r,\qquad X-Z=s,
\]
where $r,s=O(1)$. Therefore
\[
\begin{aligned}
 \Prb(X-Y=r,\ X-Z=s)
 &=
 \sum_x \Prb(X=x)\Prb(Y=x-r)\Prb(Z=x-s)\\
 &\le
 \left(\sup_j\Prb(Y=j)\right)
 \left(\sup_j\Prb(Z=j)\right)
 =O(m_C^{-1}),
\end{aligned}
\]
because, for fixed $p\in(0,1)$, a $\operatorname{Bin}(m_C,p)$ variable
has maximum point mass $O(m_C^{-1/2})$. The $A$- and $B$-conditions use
disjoint edge variables, so
\[
 \Prb(I_{uv}=I_{uw}=1)=O(n^{-2}).
\]
Since also $\Prb(I_{uv}=1),\Prb(I_{uw}=1)=O(n^{-1})$, it follows that
\[
 |\operatorname{Cov}(I_{uv},I_{uw})|
 \le
 \Prb(I_{uv}=I_{uw}=1)
 +
 \Prb(I_{uv}=1)\Prb(I_{uw}=1)
 =
 O(n^{-2}).
\]
There are $O(n^3)$ such pairs, so their total covariance contribution is
$O(n)$.

Finally, let $e=uv$ and $f=xy$ have four distinct endpoints, and expose
the four cross edges
\[
 Z=(A_{ux},A_{uy},A_{vx},A_{vy})\in\{0,1\}^4.
\]
Conditioned on $Z=z$, all remaining edge variables used by $I_e$ are
disjoint from those used by $I_f$, so the two indicators are
conditionally independent. The exposed edges only change the required
degree differences by bounded amounts. Hence, uniformly over the
sixteen possible values of $z$, Lemma~\ref{lem:half-qm} gives
\[
 \Prb(I_e=1\mid Z=z)
 =
 \frac{1+o(1)}{\pi\sqrt3\,(1-p)n},
 \qquad
 \Prb(I_f=1\mid Z=z)
 =
 \frac{1+o(1)}{\pi\sqrt3\,(1-p)n}.
\]
Therefore
\[
 \Prb(I_e=I_f=1)
 =
 \E\!\left[
   \Prb(I_e=1\mid Z)\Prb(I_f=1\mid Z)
 \right]
 =
 \frac{1+o(1)}
 {3\pi^2(1-p)^2n^2}.
\]
The one-edge estimate gives the same asymptotic for
$\Prb(I_e=1)\Prb(I_f=1)$, and hence
\[
 \operatorname{Cov}(I_e,I_f)=o(n^{-2}).
\]
There are $O(n^4)$ vertex-disjoint pairs of pairs, so their total
contribution is $o(n^2)$.

Combining the three contributions,
\[
 \operatorname{Var}(|E(H)|)=o(n^2).
\]
Chebyshev's inequality and the estimate for $\E|E(H)|$ then give
\eqref{eq:half-M-concentration}. \qed

\section{Details for the typical complement properties}
\label{app:typical-complement}

We give the details for Lemma~\ref{lem:typical-complement}. Recall that
$H\sim G(n,r)$ with $r=n^{\beta-1}$, so $nr=n^\beta$.

We use the standard Chernoff bound (see, e.g., \cite[Chapter~2]{JLR2000}): if $X$ is a sum of independent Bernoulli random variables with mean $\mu$, then for $0<\delta\leq1$,
\[
 \Prb(|X-\mu|\geq \delta\mu)
 \leq 2\exp\!\left(-\frac{\delta^2\mu}{3}\right).
\]
\paragraph{Proof of (i).}
For every $v\in V(H)$, we have $d_H(v)\sim\operatorname{Bin}(n-1,r)$ with mean
$\mu=(n-1)r=(1+o(1))n^\beta$. Taking $\delta=1$ in the Chernoff bound gives
$\Prb(d_H(v)\geq2\mu)\leq 2e^{-\mu/3}$. Hence, by a union bound over all vertices,
\[
 \Prb(\Delta(H)\geq2\mu)
 \leq 2n e^{-\mu/3}=o(1),
\]
since $\mu=\Theta(n^\beta)\gg\log n$. Thus $\Delta(H)=O(n^\beta)$ with high probability. \qed

\paragraph{Proof of (ii).}
Fix distinct $u,v$, and let $Z_{uv}$ be the number of vertices
$w\notin\{u,v\}$ adjacent to exactly one of $u$ and $v$. For each such $w$, this
occurs with probability $2r(1-r)$, independently of the other vertices. Therefore
\[
 Z_{uv}\sim\operatorname{Bin}(n-2,2r(1-r)),
 \qquad
 \mu:=\E Z_{uv}
 =2(n-2)r(1-r)
 =(2+o(1))n^\beta.
\]
Set $\delta=\sqrt{12\log n/\mu}$. Since $\mu\gg\log n$, we have
$\delta=o(1)$, and Chernoff gives
\[
 \Prb(|Z_{uv}-\mu|\geq\delta\mu)
 \leq 2e^{-\delta^2\mu/3}
 =2n^{-4}.
\]
A union bound over fewer than $n^2$ pairs shows that, with probability
$1-o(1)$, simultaneously for all distinct $u,v$,
$Z_{uv}=\mu+O(\sqrt{\mu\log n})=(2+o(1))n^\beta$.

Now the vertices outside $\{u,v\}$ contribute exactly $Z_{uv}$ to
$N_H[u]\triangle N_H[v]$. The pair $u,v$ contributes $0$ if $uv\in E(H)$ and
$2$ otherwise. Hence
\[
 |N_H[u]\triangle N_H[v]|
 =Z_{uv}+2\1[uv\notin E(H)]
 =(2+o(1))n^\beta,
\]
uniformly over all distinct $u,v$.

Finally, if $d_H(u)=d_H(v)$, then $|N_H[u]|=|N_H[v]|$, and therefore
$|N_H[u]\setminus N_H[v]|=|N_H[v]\setminus N_H[u]|$. Since their sum equals
$|N_H[u]\triangle N_H[v]|$, both are $(1+o(1))n^\beta$. This proves (ii). \qed

\paragraph{Proof of (iii).}
Split $V(H)=L\cup R$ with $|L|,|R|=n/2+O(1)$ and condition on
$H[L]$. For $v\in L$ write
\[
 d_H(v)=d_{H[L]}(v)+X_v,
 \qquad X_v\sim\operatorname{Bin}(|R|,r).
\]
The variables $(X_v)_{v\in L}$ are independent, and Stirling's formula
gives $\sup_s\Prb(X_v=s)=O(n^{-\beta/2})$. Hence, for every fixed
degree $t$,
\[
 Y_t^L:=|\{v\in L:d_H(v)=t\}|
\]
is stochastically dominated by
$\operatorname{Bin}(|L|,Cn^{-\beta/2})$. For a sufficiently large
constant $C'$, Chernoff's inequality gives, uniformly in $t$ and in the
realization of $H[L]$,
\[
 \Prb\bigl(Y_t^L>C'n^{1-\beta/2}\mid H[L]\bigr)
 \le \exp\bigl(-\Omega(n^{1-\beta/2})\bigr).
\]
A union bound over all possible $t$, and the same argument with $L$ and
$R$ interchanged, show that every degree class has size
$O(n^{1-\beta/2})$ with high probability. \qed

\paragraph{Proof of (iv).}
Choose a fixed integer $t>1+2/(1-\beta)$, and let $X_t$ count the
copies of $K_t$ in $H$. Then
\[
 \E X_t
 \le \binom nt r^{\binom t2}
 \le n^{\,t-(1-\beta)\binom t2}
 =o(1),
\]
because the exponent is negative. Markov's inequality shows that $H$
contains no $K_t$ with high probability, and hence
$\omega(H)=O_\beta(1)$. \qed

\subsection{Symmetric Lov\'asz Local Lemma}
\label{app:lll}

\begin{lemma}[Symmetric Lov\'asz Local Lemma]
\label{lem:symmetric-lll}
Let $\{A_\alpha\}_{\alpha\in\mathcal I}$ be a finite family of events.
Suppose that $\Prb(A_\alpha)\le p$ for every $\alpha\in\mathcal I$, and
that each $A_\alpha$ is independent of all but at most $D$ other events.
If
\[
 e\,p(D+1)\le1,
\]
then
\[
 \Prb\!\left(\bigcap_{\alpha\in\mathcal I}\overline{A_\alpha}\right)>0.
\]
In particular, there exists an outcome in which none of the events
$A_\alpha$ occurs.
\end{lemma}

This is the standard symmetric form of the Lov\'asz Local Lemma; see,
for example,~\cite[Chapter~5]{AlonSpencer2016}.

\end{document}